\documentclass[reqno,A4paper]{amsart}

\usepackage{amsmath}
\usepackage{amssymb}
\usepackage{amsthm}
\usepackage{enumerate}
\usepackage{mathrsfs} 
\usepackage{eqlist}
\usepackage{array}

\usepackage{times}
\usepackage{mathtools}
\usepackage{subcaption}
\usepackage{tikz}

\theoremstyle{plain}
\newtheorem{theorem}{Theorem}[section]

\newtheorem{lemma}[theorem]{Lemma}

\theoremstyle{definition}
\newtheorem{definition}{Definition}
\newtheorem{remark}{\textup{Remark}}

\numberwithin{equation}{section}

\newcommand{\N}{\mathbb{N}}
\newcommand{\setm}{\setminus}
\newcommand{\family}[2]{\left \langle #1 | #2 \right \rangle}

\newcommand{\rstr}[1]{\mathbb{#1}}
\newcommand{\A}[1]{A_{#1}}
\newcommand{\Amin}[1]{A_{#1}^{-}}
\newcommand{\rstrA}[2]{\rstr{A}_{#1,\,#2}}
\newcommand{\rstrAcarka}[2]{\rstr{A'}_{#1,\,#2}} 
\newcommand{\rstrAbar}[2]{\rstr{\bar{A}}_{#1,\,#2}} 
\newcommand{\B}[1]{B_{#1}}
\newcommand{\Bmin}[1]{B_{#1}^{-}}
\newcommand{\rstrB}[1]{\rstr{B}_{#1}}
\newcommand{\R}[3][]{R_{#2,#3}^{#1}} 
\newcommand{\vc}[1]{\mathbf{#1}}

\newcommand{\val}[2]{\mathrm{Num}_{#2}(#1)} 
\newcommand{\less}[2]{\mathrm{Num}_{#2}{(\A{#1})}}
\newcommand{\lessB}[2]{\mathrm{Num}_{#2}{(\B{#1})}}
\newcommand{\inum}[2]{\varphi_{#1}({#2})}
\newcommand{\ivect}[1]{\vc{v}_{#1}}
\newcommand{\id}{\mathrm{id}}
\newcommand{\nic}{\emptyset}
\newcommand{\hv}{\ast}
\newcommand{\relset}[1]{\mathcal{#1}}
\newcommand{\pol}[1]{\mathrm{Pol}(#1)}
\newcommand{\nuf}[1]{\mathrm{NU}(#1)}
\newcommand{\nuff}[2]{\mathrm{NU}({#1},{#2})}

\newcommand\numberthis{\addtocounter{equation}{1}\tag{\theequation}}

\newcommand{\POZN}[1]{\textcolor{red}{POZNÁMKA: #1}}

\newcommand{\MyFig}[1]{\begin{figure}[ht]#1\end{figure}}
\newcommand\relkrok{-.5}
\newcommand\relrozestup{1.7}
\newcommand\relbod{2pt}
\newcommand\reltext{3pt}
\newcommand\reltecky{-5pt}
\newcommand\reledge{}

\begin{document}

\title[The minimal arity of near unanimity polymorphisms]%
{The minimal arity of near unanimity polymorphisms}
\author[Libor Barto \and Ond\v rej Draganov]%
{Libor Barto \and Ond{\v r}ej Draganov}

\newcommand{\acr}{\newline\indent}

\address{\llap{}Department of Algebra\acr
                   Faculty of Mathematics and Physics\acr
                   Charles University\acr
                   Sokolovsk\'a 83, 18675 Praha 8\acr
                   Czech Republic}
\email{libor.barto@gmail.com, ondra@draganov.cz}



\thanks{This work was supported by the Czech Science Foundation Grant No. 13-01832S}


\begin{abstract}
Dmitriy Zhuk has proved that there exist relational structures which admit near unanimity polymorphisms, but the minimum arity of such a polymorphism is large and almost matches the known upper bounds. We present a simplified and explicit construction of such structures and a detailed, self--contained proof.
\end{abstract}

\maketitle

\section{Introduction}

An operation $t: A^n \to A$ of arity $n \geq 3$ is called a \emph{near--unanimity} (\emph{NU}, for short) operation if 
\[
t(b,a, \ldots, a) = t(a,b,a, \ldots, a) = \cdots = t(a, \ldots, a,b) = b
\]
for every $a,b \in A$. Relational structures admitting an \emph{NU polymorpism}, that is, a compatible NU operation, are interesting for numerous reasons.
We refer the reader to~\cite{Bar13,BKW17,FV98,LLZ05,M09,Zh14} and references therein. 

The first author~\cite{Bar13} and, independently, D. Zhuk~\cite{Zh14} have proved that it is decidable whether a finite relational structure $\rstr{A}$ admits a near--unanimity polymorphism. Both proofs give an upper bound on the smallest arity of an NU polymorphism. A refinement of these results from~\cite{BB17} (generalizing~\cite{Bul14}) proves the following upper bound.

\begin{theorem} \label{thm:upper}
If a relational structure $\rstr{A}$ with universe of size $n \geq 2$ admits an NU polymorphism, then it admits one of arity
\[
\frac12 (2m-2)^{3^n} + 1,
\]
where $m \geq 2$ is the maximum arity of a relation in $\rstr{A}$.
\end{theorem}

Zhuk~\cite{Zh14} has also constructed relational structures witnessing that this upper bound is essentially optimal:

\begin{theorem} \label{thm:lower}
For each $m \geq 3$ and $n \geq 2$ (resp. $m=2$, $n \geq 3$), there exists a relational structure $\rstr{A}$ with universe of size $n$ and relations of arity at most $m$ such that $\rstr{A}$ admits an NU polymorphism, but no NU polymorphism of arity less than or equal to 
\[
(m-1)^{2^{n-2}} \quad (\mbox{resp. } 2^{2^{n-3}})
\]
\end{theorem}

The lower bounds in Theorem~\ref{thm:lower} indicate that the combinatorial core of the proofs of Theorem~\ref{thm:upper} in \cite{Bar13,Zh14,BB17} cannot be circumvent. It is interesting to compare this with the situation for weak near--unanimity operations, where the upper bound does not depend on $m$ and the original huge upper bound in terms of $n$ that follows from~\cite{MM08} was eventually pushed down to linear~\cite{BK12}. Theorem~\ref{thm:lower} also shows that feasibility results that depend on the arity of NU polymorphism are, for some structures, of purely theoretical interest; for instance, the  bounded strict width algorithm for fixed--template constraint satisfaction problems~\cite{FV98}.

These are some of the reasons that make Theorem~\ref{thm:lower} interesting. 
Unfortunately,
the proof of Theorem~\ref{thm:lower} in~\cite{Zh14} is rather involved:
The relational structures contain a lot of $m$-ary relations (about $2^n$) and they are defined in a complicated way from a matrix, which is recursively constructed. 
Moreover, the computation of the minimal arity of an NU polymorphisms is quite sketchy and requires some concepts from the more complex first part of the paper, which proves a version of Theorem~\ref{thm:upper}.

In this paper we present a simplified and explicit construction of structures from Theorem~\ref{thm:lower} and a detailed, self--contained proof. 

We require only basic knowledge of universal algebra~\cite{BS81,Berg11}. Let us just recall that a polymorphism of a structure $\rstr{A}$ is compatible with every relation that is \emph{primitively positively definable} (\emph{pp--definable}, for short) from $\rstr{A}$, that is, definable from $\rstr{A}$ by a first order formula which uses only the equality, conjunction, and existential quantification.

\section{Relations of arity higher than two}

In this section we prove Theorem~\ref{thm:lower} for $m \geq 3$ and $n \geq 2$. The case $m=2$ requires a slight modification and is dealt with in the next section.

It will be convenient to slightly modify the parameters $m,n$. For the whole section we fix $n \in \N_0$ and $m \in \N$, $m \geq 2$. We will construct (Subsection~\ref{sec:structureA}) a relational structure $\rstr{A}$ with an $(n+2)$-element universe whose relations have arity at most $(m+1)$, prove
that it has no NU polymorphism of arity $m^{2^n}$ (Subsection~\ref{sec:non-existence}), and construct an NU polymorphism of arity $m^{2^n}+1$ (Subsection~\ref{sec:construction}). 

\subsection{Construction} \label{sec:structureA}

The structure $\rstr{A}$ consists of all (nonempty) unary relations and $n+1$ relations of arity $m+1$:

\[
\rstr{A}=(A;S_0,S_1,\dots,S_n,\family{X}{\emptyset\neq X\subseteq \A{n}}),
\]
where
\begin{align*}
A &= \{a,0,1, \dots,n\} \\
S_i&=\left(\{a,0,\dots,i-1\}\times \{a,i\}^m\ \setm\ \{(\underbrace{a,i,\dots,i}_{m+1})\}\right)\\
& \quad \cup \{(\underbrace{i+1,\dots,i+1}_{m+1}),\dots,(\underbrace{n,\dots,n}_{m+1})\}.
\end{align*}

\subsection{Notation and useful facts}
The following notation will be useful. 

For $i\in \N_0$ we denote
\[
A_{i+1} = \{a,0,1, \dots, i\}, \quad
\Amin{i+1}=\A{i+1}\setm\{a\}=\{0,\dots,i\},
\]
that is, $A_0 = \{a\}$ and $A = A_{n+1}$. We equip each $A_i$ with the linear order $a < 0 < 1 < \cdots <i$.

Moreover, for $i \in \{0,\ldots,n\}$, let $R_i$ be the projection of $S_i$ onto the first two coordinates. Explicitly (see Figure~\ref{fig:R}):
\[
R_i=\left(\{a,0,\dots,i-1\}\times \{a,i\}\right) \cup \{(i+1,i+1),\dots,(n,n)\}.
\]
\MyFig{ 
	\centering
	\begin{subfigure}{.5\textwidth}
		\centering
		\begin{tikzpicture}[
bod/.style={radius=\relbod},
textL/.style={left=\reltext},
teckyL/.style={left=\reltecky},
textR/.style={right=\reltext},
teckyR/.style={right=\reltecky},
edge/.style={\reledge}
]

\foreach \x/\ozn in {0/a, 1/0, 2/1, 3/t1, 4/i-1, 5/i, 6/i+1, 7/i+2, 8/t2, 9/n}
	\coordinate (l\ozn) at (0,{\x*(\relkrok)});
\foreach \x/\ozn in {0/a, 1/0, 2/1, 3/t1, 4/i-1, 5/i, 6/i+1, 7/i+2, 8/t2, 9/n}
	\coordinate (r\ozn) at ({0+\relrozestup},{\x*(\relkrok)});

\filldraw[fill=black, draw=black]
(la) circle[bod] node[textL] {$a$}
(l0) circle[bod] node[textL] {0}
(l1) circle[bod] node[textL] {1}
(lt1) node[teckyL] {$\vdots$}
(li-1) circle[bod] node[textL] {$i-1$}
(li) circle[bod] node[textL] {$i$}
(li+1) circle[bod] node[textL] {$i+1$}
(li+2) circle[bod] node[textL] {$i+2$}
(lt2) node[teckyL] {$\vdots$}
(ln) circle[bod] node[textL] {$n$};

\filldraw[fill=black, draw=black]
(ra) circle[bod] node[textR] {$a$}
(r0) circle[bod] node[textR] {0}
(r1) circle[bod] node[textR] {1}
(rt1) node[teckyR] {$\vdots$}
(ri-1) circle[bod] node[textR] {$i-1$}
(ri) circle[bod] node[textR] {$i$}
(ri+1) circle[bod] node[textR] {$i+1$}
(ri+2) circle[bod] node[textR] {$i+2$}
(rt2) node[teckyR] {$\vdots$}
(rn) circle[bod] node[textR] {$n$};

\draw[edge] (la)--(ri);
\draw[edge] (l0)--(ri);
\draw[edge] (l1)--(ri);
\draw[edge] (li-1)--(ri);

\draw[edge] (la)--(ra);
\draw[edge] (l0)--(ra);
\draw[edge] (l1)--(ra);
\draw[edge] (li-1)--(ra);

\draw[edge] (li+1)--(ri+1);
\draw[edge] (li+2)--(ri+2);
\draw[edge] (ln)--(rn);
		\end{tikzpicture}
	\caption{Relation $R_i$ on $\A{n+1}$}
	\end{subfigure}
	\begin{subfigure}{.5\textwidth}
		\centering
		\begin{tikzpicture}[
bod/.style={radius=\relbod},
textL/.style={left=\reltext},
teckyL/.style={left=\reltecky},
textR/.style={right=\reltext},
teckyR/.style={right=\reltecky},
edge/.style={\reledge}
]

\foreach \x/\ozn in {0/a, 1/0, 2/1, 3/2, 4/3}
	\coordinate (l\ozn) at (0,{\x*(\relkrok)});
\foreach \x/\ozn in {0/a, 1/0, 2/1, 3/2, 4/3}
	\coordinate (r\ozn) at ({0+\relrozestup},{\x*(\relkrok)});

\filldraw[fill=black, draw=black]
(la) circle[bod] node[textL] {$a$}
(l0) circle[bod] node[textL] {0}
(l1) circle[bod] node[textL] {1}
(l2) circle[bod] node[textL] {2}
(l3) circle[bod] node[textL] {3};

\filldraw[fill=black, draw=black]
(ra) circle[bod] node[textR] {$a$}
(r0) circle[bod] node[textR] {0}
(r1) circle[bod] node[textR] {1}
(r2) circle[bod] node[textR] {2}
(r3) circle[bod] node[textR] {3};

\draw[edge] (la)--(r2);
\draw[edge] (l0)--(r2);
\draw[edge] (l1)--(r2);

\draw[edge] (la)--(ra);
\draw[edge] (l0)--(ra);
\draw[edge] (l1)--(ra);

\draw[edge] (l3)--(r3);
		\end{tikzpicture}
	\caption{Example -- relation $R_2$ on $\A{4}$}
	\end{subfigure}
\caption{Relation $R_i$}
\label{fig:R}
} 







Finally, for $i \in \{1, \ldots, n\}$, we denote by $\A{i}|i|\dots|n$ or $a0\dots i-1|i|\dots|n$ the equivalence on $A$ with blocks
 $\{a,0,\dots,i-1\},\{i\},\{i+1\},\dots,\{n\}$. 

\begin{lemma}\label{lem:congruence}
Every polymorphism of $\rstr{A}$ is compatible with $\A{i}|i|\dots|n$ for every $i\in\{1,\dots,n\}$
\end{lemma}

\begin{proof}
We claim that 
\begin{equation} \label{eq:cong}
\A{i}|i|\dots|n = 
R^{-1}_0 \circ R^{-1}_1 \circ \dots \circ R^{-1}_{i-1} \circ R_{i-1} \circ R_{i-2} \circ \dots \circ R_1 \circ R_0.
\end{equation}
We denote the right hand side by $\alpha$.



\MyFig{
	\centering
	\begin{tikzpicture}[
bod/.style={radius=\relbod},
textL/.style={left=\reltext},
teckyL/.style={left=\reltecky},
textR/.style={right=\reltext},
teckyR/.style={right=\reltecky},
edge/.style={\reledge}
]

\newcommand\rozestup{(\textwidth)/11}
\newcommand\vysPopisku{0.7}

\foreach \x/\ozn in {0/a, 1/0, 2/1, 3/2, 4/t1, 5/i-3, 6/i-2, 7/i-1, 8/i, 9/t2, 10/n}
	\coordinate (l\ozn) at (0,{\x*(\relkrok)});
\foreach \x/\ozn in {0/a, 1/0, 2/1, 3/2, 4/t1, 5/i-3, 6/i-2, 7/i-1, 8/i, 9/t2, 10/n}
	\coordinate (r\ozn) at ({9*(\rozestup)},{\x*(\relkrok)});
\foreach \x in {1,...,8}{
	\foreach \y in {0,...,10}
		\coordinate ({m\x\y}) at ({\x*(\rozestup)},{\y*(\relkrok)});
	\foreach \y in {0,...,3,5,6,7,8,10}
		\filldraw ({m\x\y}) circle[bod];
	\draw (m\x4) node {$\vdots$};
	\draw (m\x9) node {$\vdots$};
}

\filldraw[fill=black, draw=black]
(la) circle[bod] node[textL] {$a$}
(l0) circle[bod] node[textL] {0}
(l1) circle[bod] node[textL] {1}
(l2) circle[bod] node[textL] {2}
(lt1) node[teckyL] {$\vdots$}
(li-3) circle[bod] node[textL] {$i-3$}
(li-2) circle[bod] node[textL] {$i-2$}
(li-1) circle[bod] node[textL] {$i-1$}
(li) circle[bod] node[textL] {$i$}
(lt2) node[teckyL] {$\vdots$}
(ln) circle[bod] node[textL] {$n$};

\filldraw[fill=black, draw=black]
(ra) circle[bod] node[textR] {$a$}
(r0) circle[bod] node[textR] {0}
(r1) circle[bod] node[textR] {1}
(r2) circle[bod] node[textR] {2}
(rt1) node[teckyR] {$\vdots$}
(ri-3) circle[bod] node[textR] {$i-3$}
(ri-2) circle[bod] node[textR] {$i-2$}
(ri-1) circle[bod] node[textR] {$i-1$}
(ri) circle[bod] node[textR] {$i$}
(rt2) node[teckyR] {$\vdots$}
(rn) circle[bod] node[textR] {$n$};

\foreach \x/\popis in {0/$R_0^{-1}$, 1/$R_1^{-1}$, 2/$R_2^{-1}$, 3/$\dots$, 4/$R_{i-1}^{-1}$, 5/$R_{i-1}$, 6/$R_{i-2}$, 7/$\dots$, 8/$R_0$}
	\draw ({(\rozestup)/2 + \x*(\rozestup)},\vysPopisku) node {\popis};


\draw[edge] (la)--(m10);
\draw[edge] (l0)--(m10);

\draw[edge] (l1)--(m12);
\draw[edge] (l2)--(m13);
\draw[edge] (li-3)--(m15);
\draw[edge] (li-2)--(m16);
\draw[edge] (li-1)--(m17);
\draw[edge] (li)--(m18);
\draw[edge] (ln)--(m110);

\draw[edge] (m10)--(m20);
\draw[edge] (m10)--(m21);
\draw[edge] (m12)--(m20);
\draw[edge] (m12)--(m21);

\draw[edge] (m13)--(m23);
\draw[edge] (m15)--(m25);
\draw[edge] (m16)--(m26);
\draw[edge] (m17)--(m27);
\draw[edge] (m18)--(m28);
\draw[edge] (m110)--(m210);

\draw[edge] (m20)--(m30);
\draw[edge] (m20)--(m31);
\draw[edge] (m20)--(m32);
\draw[edge] (m23)--(m30);
\draw[edge] (m23)--(m31);
\draw[edge] (m23)--(m32);

\draw[edge] (m25)--(m35);
\draw[edge] (m26)--(m36);
\draw[edge] (m27)--(m37);
\draw[edge] (m28)--(m38);
\draw[edge] (m210)--(m310);

\draw ({(\rozestup)/2 + 3*(\rozestup)},{5*\relkrok}) node {$\dots$};

\draw[edge] (m40)--(m50);
\draw[edge] (m40)--(m51);
\draw[edge] (m40)--(m52);
\draw[edge] (m40)--(m53);
\draw[edge] (m40)--(m55);
\draw[edge] (m40)--(m56);
\draw[edge] (m47)--(m50);
\draw[edge] (m47)--(m51);
\draw[edge] (m47)--(m52);
\draw[edge] (m47)--(m53);
\draw[edge] (m47)--(m55);
\draw[edge] (m47)--(m56);

\draw[edge] (m48)--(m58);
\draw[edge] (m410)--(m510);

\draw[edge] (m50)--(m60);
\draw[edge] (m51)--(m60);
\draw[edge] (m52)--(m60);
\draw[edge] (m53)--(m60);
\draw[edge] (m55)--(m60);
\draw[edge] (m56)--(m60);
\draw[edge] (m50)--(m67);
\draw[edge] (m51)--(m67);
\draw[edge] (m52)--(m67);
\draw[edge] (m53)--(m67);
\draw[edge] (m55)--(m67);
\draw[edge] (m56)--(m67);

\draw[edge] (m58)--(m68);
\draw[edge] (m510)--(m610);

\draw[edge] (m60)--(m70);
\draw[edge] (m61)--(m70);
\draw[edge] (m62)--(m70);
\draw[edge] (m63)--(m70);
\draw[edge] (m65)--(m70);
\draw[edge] (m60)--(m76);
\draw[edge] (m61)--(m76);
\draw[edge] (m62)--(m76);
\draw[edge] (m63)--(m76);
\draw[edge] (m65)--(m76);

\draw[edge] (m67)--(m77);
\draw[edge] (m68)--(m78);
\draw[edge] (m610)--(m710);

\draw ({(\rozestup)/2 + 7*(\rozestup)},{5*\relkrok}) node {$\dots$};

\draw[edge] (m80)--(ra);
\draw[edge] (m80)--(r0);

\draw[edge] (m82)--(r1);
\draw[edge] (m83)--(r2);
\draw[edge] (m85)--(ri-3);
\draw[edge] (m86)--(ri-2);
\draw[edge] (m87)--(ri-1);
\draw[edge] (m88)--(ri);
\draw[edge] (m810)--(rn);

	\end{tikzpicture}
	\caption{Composition $R^{-1}_0 \circ R^{-1}_1 \circ \dots \circ R^{-1}_{i-1} \circ R_{i-1} \circ R_{i-2} \circ \dots \circ R_1 \circ R_0$} \label{obr:slozeni}
} 



To prove the inclusion ``$\supseteq$'', we just need to verify that $(r,r'),(r',r)\notin \alpha_i$ for every $r, r'\in \A{n+1}$, $r\geq i$, $r\neq r'$. This holds since $(r,r'')\notin R_j$ for all $r''\neq r$ and all $j\in \{0,\dots,i-1\}$.

To prove the reverse inclusion, we first check that $(r,r)\in \alpha_i$ for all $r\in \A{n+1}$, $r\geq i$. This is true since $(r,r)\in R_j$ for every $j\in \{0,\dots,i-1\}$.

Now we need to show that any pair $(r,s)$, $r,s \in \{a,0,\dots,i-1\}$, is an element of $\alpha$.
This is best seen in the picture:  we go through the graph in Figure \ref{obr:slozeni} using $r\textnormal{--}r$ edges until we get to the point where we can ``jump'' through an $t\textnormal{--}a$ edge to $a$. Then we go through $a\textnormal{--}a$ edges until we can ``jump'' to $s$ using $a\textnormal{--}s$ edge. Then we go to the end through $s\textnormal{--}s$ edges.


Equation~\eqref{eq:cong} implies that $\A{i}|i|\dots|n$ is pp-definable from $R_i's$, which are pp-definable from $\rstr{A}$, and the claim follows.
\end{proof}

For a set $X$ we write $(X,\dots,X)$ to mean ``any vector $(x_1,\dots,x_l)$ such that $x_1,\dots,x_l\in X$''. For example, 
\[
t(\underbrace{\A{i},\dots,\A{i}}_l,i+1,\dots,i+1,\dots,n,\dots,n)\neq a,
\]
means\[
t(x_1,\dots,x_l,i+1,\dots,i+1,\dots,n,\dots,n)\neq a,
\]
for any $x_1,\dots,x_l\in \A{i}$.

The compatibility with $S_i$ will be used as follows.

\begin{lemma} \label{lem:comp}
If $t$ is a polymorphism of $\rstr{A}$ and $i \in \{0, \dots, n\}$, then 
\[
t(\underbrace{a,\dots,a}_{m\cdot l_a},\underbrace{\Amin{i},\dots,\Amin{i}}_{l_i},\underbrace{i+1,\dots,i+1}_{l_{i+1}},\dots,\underbrace{n,\dots,n}_{l_n})\neq a,
\]
whenever
\[
\left. \begin{split}
t(a,\dots,a,i,\dots,i,\dots,i,\dots,i,i,\dots,i,i+1,\dots,i+1,\dots,n,\dots,n)&=i \\
t(i,\dots,i,a,\dots,a,\dots,i,\dots,i,i,\dots,i,i+1,\dots,i+1,\dots,n,\dots,n)&=i \\
\vdots \hspace{11cm} \vdots \\
t(\underbrace{i,\dots,i}_{l_a},\underbrace{i,\dots,i}_{l_a},\dots,\underbrace{a,\dots,a}_{l_a},\underbrace{i,\dots,i}_{l_i},\underbrace{i+1,\dots,i+1}_{l_{i+1}},\dots,\underbrace{n,\dots,n}_{l_n})&= i  \\
\end{split} \right\} m
\]
for arbitrary $l_a,l_i,\dots,l_n\in \N_0$ such that $m\cdot l_a+l_i+\dots+l_n=k$. (Lemma also holds for any permutation of arguments, the same for every row.)
\end{lemma}
\begin{proof}
The ``columns''
\begin{align*}
&(a,a,i,\dots,i),(a,i,a,i,\dots,i),\dots,(a,i,\dots,i,a),(\Amin{i},i,\dots,i), \\
&(i+1,\dots,i+1),\dots,(n,\dots,n)
\end{align*}
are in $S_i$. 
Since $t$ is compatible with $S_i$, then $t$ applied to the columns  has to be in $S_i$ as well, i.e., it can not be $(a,i,\dots,i)$. Hence\[
t(\underbrace{a,\dots,a}_{m\cdot l_a},\Amin{i},\dots,\Amin{i},i+1,\dots,i+1,\dots,n,\dots,n)\neq a.
\]
\end{proof}

\subsection{Example} \label{subsec:example}

We will show in the next subsection that $\rstr{A}$ admits no NU polymorphism of arity $m^{2^n}$ (and thus no NU polymorphism of smaller arity).
Since the formal proof is a bit technical, we first illustrate the idea in the case $n=m=3$, thus
\[
\A{4}=\{a,0,1,2,3\},\; \rstrA{4}{4}=(\A{4};S_0,S_1,S_2,S_3,\family{D_X}{X\subseteq \A{4}}),
\]
$S_i$ are $4$-ary. 

Let $t$ be an NU polymorphism of $\rstr{A}$ of arity $3^{2^3}=3^8$. Step by step we show that $t(a,\dots,a)\neq a$, which is a contradiction since $t$ should preserve the unary relation $\{a\}$. In each step we  choose a suitable tuple $\vc{v}_k$ so that  $t(\vc{v}_{k-1})\neq a$ implies $t(\vc{v}_k)\neq a$ and the number of initial $a$'s is 3 (=$m$) times larger than in $\vc{v}_{k-1}$. The choice of these vectors in the general situation is described in Definition \ref{def:vectors}.

Since $t$ is an NU term, we have
\begin{align*}
t(a,3,3,3,\dots,3)&=3,\\
t(3,a,3,3,\dots,3)&=3,\\
t(3,3,a,\underbrace{3,\dots,3}_{3^8-3})&=3.\\
\end{align*}
By Lemma \ref{lem:comp}  (with $i=3$, $l_a=1$, $l_3=3^8-3$) we get\[
t(a,a,a,\Amin{3},\dots,\Amin{3})\neq a.
\]
By the same argument, this inequality also holds for any permutation of the arguments.

In particular, we get
	\begin{equation}\label{pr:krok0}
t(\vc{v}_0)=t(\underbrace{a,a,a}_3,\underbrace{0,0,0,0,0,0}_{3^{2}-3},\underbrace{1,\dots,1}_{3^4-3^2},\underbrace{2,\dots,2}_{3^8-3^4})\neq a
	\end{equation}
and similarly for any permutation of the arguments.

We now want to show that
	\begin{equation}\label{pr:krok1}
t(\vc{v}_1)=t(\underbrace{a,\dots,a}_{3^2},\underbrace{1,\dots,1}_{3^4-3^2},\underbrace{2,\dots,2}_{3^8-3^4})\neq a.
	\end{equation}
Since $t$ is compatible with the unary relation ${\{a,0,1,2\}}$ and $t(\vc{v}_0)\neq a$, then $t(\vc{v}_0)\in\{0,1,2\}$. If $t(\vc{v}_0)\in\{1,2\}$ then also $t(\vc{v}_1)\in\{1,2\}$, because $t$ is compatible with ${a0|1|2|3}$ by Lemma \ref{lem:congruence}. In that case the inequality \eqref{pr:krok1} obviously holds.

Otherwise we have $t(\vc{v}_0)=0$. In that case we use the compatibility with  ${a0|1|2|3}$ to obtain
	\begin{align*}
t(\underbrace{\A{1},\dots,\A{1}}_{3^2},{\underbrace{1,\dots,1}_{3^4-3^2}},{\underbrace{2,\dots,2}_{3^8-3^4}})&\in \A{1}=\{a,0\}
	\end{align*}
and thus, by (\ref{pr:krok0}) for suitable permutations, we get
	\begin{align*}
t(a,a,a,0,0,0,0,0,0,1,\dots,1,2,\dots,2)&=0, \\
t(0,0,0,a,a,a,0,0,0,1,\dots,1,2,\dots,2)&=0, \\
t(0,0,0,0,0,0,a,a,a,\underbrace{1,\dots,1}_{3^4-3^2},\underbrace{2,\dots,2}_{3^8-3^4})&=0.
	\end{align*}
Inequality \eqref{pr:krok1} now follows from Lemma~\ref{lem:comp} (with $i=0$, $l_a=3$, $l_0=0$, $l_1=3^4-3^2$, $l_2=3^8-3^4$, $l_3=0$). 

The next step is to show
	\begin{equation}\label{pr:krok2}
t(\vc{v}_2)=t(\underbrace{a,\dots,a}_{3^3},\underbrace{0,\dots,0}_{3^4-3^3},\underbrace{2,\dots,2}_{3^8-3^4})\neq a.
	\end{equation}
We prove this similarly as before using inequality (\ref{pr:krok1}). We already now that $t(\vc{v}_1)\in \{1,2\}$. If $t(\vc{v}_1)=2$, then $t(\vc{v}_2)=2$ by the compatibility with  ${a01|2|3}$, and~(\ref{pr:krok2}) is again immediate. For the case $t(\vc{v}_1)=1$ we use the compatibility with ${a01|2|3}$, compatibility of $t$ with ${\{a,1,2\}}$ and inequality (\ref{pr:krok1}) for suitable permutations of the arguments to get
	\begin{align*}
t(a,\dots,a,1,\dots,1,1,\dots,1,1,\dots,1,2,\dots,2)&=1, \\
t(1,\dots,1,a,\dots,a,1,\dots,1,1,\dots,1,2,\dots,2)&=1, \\
t({\underbrace{1,\dots,1}_{3^2}},{\underbrace{1,\dots,1}_{3^2}},{\underbrace{a,\dots,a}_{3^2}},{\underbrace{1,\dots,1}_{3^4-3\cdot3^2}},{\underbrace{2,\dots,2}_{3^8-3^4}})&=1.
	\end{align*}
Inequality (\ref{pr:krok2}) now follows from Lemma \ref{lem:comp}.

In a similar way we gradually prove the following inequalities:
	\begin{align}
t(\underbrace{a,\dots,a}_{3^4},\underbrace{2,\dots,2}_{3^8-3^4})&\neq a, \\
t(\underbrace{a,\dots,a}_{3^5},\underbrace{0,\dots,0}_{3^6-3^5},\underbrace{1,\dots,1}_{3^8-3^6})&\neq a, \\
t(\underbrace{a,\dots,a}_{3^6},\underbrace{1,\dots,1}_{3^8-3^6})&\neq a, \\
t(\underbrace{a,\dots,a}_{3^7},\underbrace{0,\dots,0}_{3^8-3^7})&\neq a, \\
t(\underbrace{a,\dots,a}_{3^8})&\neq a. \label{pr:krok7}
	\end{align}
Inequality (\ref{pr:krok7}) gives us a contradiction.

\subsection{Non-existence of NU polymorphisms of low arity} \label{sec:non-existence}

We now proof that $\rstr{A}$ does not admit an NU polymorphism $t$ of arity $m^{2^n}$ in general. For a contradiction, assume that $t$ is such a polymorphism. Note that the case $n=0$, $m=2$ is not interesting, so let us assume that $n >0$ or $m>2$. 





We first introduce a useful notation.
Let $x_1,\dots,x_k\in A$, $\vc{x}=(x_1,\dots,x_k)$ and let $i\in A$. We will denote by $\val{i}{\vc{x}}$ the number of occurrences of the element~$i$ in the tuple~$\vc{x}$ and by~$\less{i}{\vc{x}}$ the number of occurrences of the elements that are strictly less than~$i$ in~$\vc{x}$. Here we also allow $i=n+1$ to denote $\less{n+1}{\vc{x}}=k$.

For $b_0,\dots,b_l\in \{0,1\}, l\in \N$, we denote by ${\overline{b_l\dots b_0}}$ the integer with binary representation ${b_l\dots b_0}$.

\begin{definition} \label{def:vectors}
Let $k\in \{0,1,\dots,2^n-1\}$ be an integer with binary representation $b_{n-1}b_{n-2}\dots b_1b_0$. For $i\in \{a,0,\dots,n-1\}$ we define \[
\inum{k}{i}=
	\begin{cases}
m^{k+1}, & \text{for $i=a$}, \\
(1-b_i)\cdot m^{\overline{b_{n-1}\dots b_{i+1}0\dots 0}+2^i} \cdot (m^{2^i}-1), & \text{for $i\in \{0,\dots,n-1\}$},
	\end{cases}
\]
where there are $i+1$ zeros after $b_{i+1}$ in the exponent. 

Moreover, we put \[
\ivect{k}=(\underbrace{a,\dots,a}_{\inum{k}{a}},\underbrace{0,\dots,0}_{\inum{k}{0}},\dots,\underbrace{n-1,\dots,n-1}_{\inum{k}{n-1}}).
\]
\end{definition}

For instance, \[
\ivect{0}=(\underbrace{a,\dots,a}_{m},\underbrace{0,\dots,0}_{m\cdot(m-1)},\underbrace{1,\dots,1}_{m^2\cdot(m^2-1)},\dots,\underbrace{i,\dots,i}_{\mathclap{m^{2^i}\cdot(m^{2^i}-1)}},\dots,\underbrace{n-1,\dots,n-1}_{m^{2^{n-1}}\cdot(m^{2^{n-1}}-1)}).
\]
Also note that $\val{i}{\ivect{k}}=\inum{k}{i}$.

We will need the following technical lemma.

\begin{lemma}\label{lem:i1}
Let $k\in \{0,1,\dots,2^n-1\}$ be an integer with binary representation ${b_{n-1}b_{n-2}\dots b_1b_0}$. Let $i\in \{0,\dots,n-1\}$ be the least index such that $b_i=0$, thus $k=\overline{b_{n-1}\dots b_{i+1} 0 \underbrace{1 \dots 1}_i}$ and $k+1=\overline{b_{n-1}\dots b_{i+1} 1 \underbrace{0 \dots 0}_{i}}$. Then
	\begin{enumerate}[a)]
\item \label{lem:i1a}
$\inum{k}{j}=0$\ \ for every $j\in \{0,1,\dots,i-1\}$,

\item \label{lem:i1b}
$\inum{k}{i}=m^{k+1}\cdot(m^{2^i}-1)$,

$\less{i+1}{\ivect{k}}=\inum{k}{a}+\inum{k}{i}=m^{k+1+2^i}$,

\item \label{lem:i1c}
$\inum{k+1}{i}=0$,

$\inum{k+1}{j}=\inum{k}{j}$\ \ for every $j\in \{i+1,\dots,n-1\}$,

\item \label{lem:i1d}
$\less{i}{\ivect{k+1}}=\inum{k+1}{a}+\inum{k+1}{0}+\dots+\inum{k+1}{i-1}=m^{k+1+2^i}$.
	\end{enumerate}
\end{lemma}
\begin{proof}
	\begin{enumerate}[a)]
\item Follows from the definitions.

\item By the assumption, we have
\begin{align*}
k&=\overline{b_{n-1}\dots b_{i+1}0\underbrace{1\dots1}_{i}}=\overline{b_{n-1}\dots b_{i+1}0\underbrace{0\dots0}_{i}}+\overline{0\dots0\underbrace{1\dots1}_{i}}= \\
&=\overline{b_{n-1}\dots b_{i+1}\underbrace{0\dots0}_{i+1}}+2^i-1.
\end{align*}

Then, by the definition,
\[
\inum{k}{i}=(1-b_i)\cdot m^{\overline{b_{n-1}\dots b_{i+1}0\dots 0}+2^i} \cdot (m^{2^i}-1)=m^{k+1}\cdot(m^{2^i}-1).
\]
The second part follows from \ref{lem:i1a}) and an easy calculation \[\inum{k}{a}+\inum{k}{i}=m^{k+1}+m^{k+1}\cdot(m^{2^i}-1)=m^{k+1+2^i}.
\]

\item Follows from the definition.

\item Since $k+1=\overline{b_{n-1}\dots b_{i+1} 1 \underbrace{0 \dots 0}_{i}}$, 
then
		\begin{multline*}
\inum{k+1}{a}+\inum{k+1}{0}+\dots+\inum{k+1}{i-1}=\\
=m^{k+1+1}+m^{k+1+2^0}\cdot(m-1)+m^{k+1+2^1}\cdot(m^2-1)+\dots+m^{k+1+2^{i-1}}\cdot(m^{2^{i-1}}-1)=\\
=m^{k+2}+m^{k+3}-m^{k+2}+m^{k+5}-m^{k+3}+\dots+m^{k+1+2^{i-1}+2^{i-1}}-m^{k+1+2^{i-1}}=\\
=m^{k+1+2^{i}}.
		\end{multline*}
	\end{enumerate}
\end{proof}

For $k\in \{0,1,\dots,2^n-1\}$ and any permutation $\sigma$ of $\{1,2,\dots,m^{2^n}\}$ we will now prove by induction on $k$ that \[
t(\sigma(\ivect{k}))\neq a,
\]
where by $\sigma((x_1,x_2,\dots,x_{m^{2^n}}))$ we mean the tuple $(x_{\sigma(1)},x_{\sigma(2)},\dots,x_{\sigma(m^{2^n})})$. Without loss of generality we will prove the claim only for the case $\sigma=\id$ -- for a general permutation $\sigma$ the induction step can be repeated with all the arguments permuted by $\sigma$.

The base case $k=0$ is a consequence of $t$ being NU. We have
\begin{align*}
t(a,n,n,n,\dots,n)&=n\\
t(n,a,n,n,\dots,n)&=n\\
\vdots\\
t(\underbrace{n,n,\dots,n,a}_{m},n,\dots,n)&=n.
\end{align*}
By Lemma \ref{lem:comp} with $i=n$,
\[
t(\underbrace{a,\dots,a}_{m},\Amin{n},\dots,\Amin{n})\neq a.
\]
In particular, \[
t(\ivect{0})\neq a.
\]


We will now prove the induction step. Let $k\in \{0,1,\dots,2^n-2\}$ and suppose  that $t(\sigma(\ivect{k}))\neq a$ for every permutation $\sigma$. We will show that $t(\ivect{k+1})\neq a$.

Let $b_{n-1}\dots b_0$ be the binary representation of $k$ and $i\in \{0,\dots,n-1\}$ be the least index such that $b_i=0$. Then, by Lemma \ref{lem:i1},
\begin{align*}
t(&\ivect{k})=t(\underbrace{\overbrace{a,\dots,a}^{m^{k+1}},i,\dots,i}_{m^{k+1+2^i}},\underbrace{i+1,\dots,i+1}_{\inum{k}{i+1}},\dots,\underbrace{n-1,\dots,n-1}_{\inum{k}{n-1}}),\\
t(&\ivect{k+1})=\\
&t(\underbrace{\overbrace{a,\dots,a}^{m^{k+2}},0,\dots,0,\dots,i-1,\dots,i-1}_{m^{k+1+2^i}},\underbrace{i+1,\dots,i+1}_{\inum{k}{i+1}},\dots,\underbrace{n-1,\dots,n-1}_{\inum{k}{n-1}}).
\end{align*}
As $t$ is compatible with the unary relation ${\{a,i,i+1,\dots,n-1\}}$, we get $t(\ivect{k})\in{\{a,i,i+1,\dots,n-1\}}$. We distinguish two cases.

If $t(\ivect{k})=j\in{\{i+1,\dots,n-1\}}$, then $t(\ivect{k+1})=j$ since $a01\dots i|i+1|\dots|n$ is compatible with $t$ (see Lemma~\ref{lem:congruence}). In particular, $t(\ivect{k+1})\neq a$.

Otherwise $t(\ivect{k})\in{\{a,i\}}$. The compatibility with $a01\dots i|i+1|\dots|n$ and ${\{a,i,i+1,\dots,n-1\}}$ yields
\[
t(\mu(\underbrace{\overbrace{a,\dots,a}^{m^{k+1}},i,\dots,i}_{m^{k+1+2^i}}),\underbrace{i+1,\dots,i+1}_{\inum{k}{i+1}},\dots,\underbrace{n-1,\dots,n-1}_{\inum{k}{n-1}})\in{\{a,i\}}
\]
for any permutation $\mu$. By the induction hypothesis, $t(\sigma(\ivect{k}))\neq a$ for any permutation $\sigma$, and so\[
t(\mu(\underbrace{\overbrace{a,\dots,a}^{m^{k+1}},i,\dots,i}_{m^{k+1+2^i}}),\underbrace{i+1,\dots,i+1}_{\inum{k}{i+1}},\dots,\underbrace{n-1,\dots,n-1}_{\inum{k}{n-1}})=i
\]
for any $\mu$.
Lemma \ref{lem:comp} now yields \[
t(\underbrace{\overbrace{a,\dots,a}^{m\cdot m^{k+1}},\Amin{i},\dots,\Amin{i}}_{m^{k+1+2^i}},\underbrace{i+1,\dots,i+1}_{\inum{k}{i+1}},\dots,\underbrace{n-1,\dots,n-1}_{\inum{k}{n-1}})\neq a,
\]
so, in particular, $t(\ivect{k+1})\neq a$, which is our claim.

We have shown that $t(\ivect{k})\neq a$ for every $k\in{\{0,1,\dots,2^n-1\}}$, in particular \[
t(\ivect{2^n-1})=t(\underbrace{a,\dots,a}_{m^{2^n}})\neq a.
\]
This contradicts the fact that $t$ preserves $\{a\}$ and finishes the proof.

\begin{remark}
The proof actually shows that $\rstr{A}$ has no polymorphism $t$ of arity $m^{2^n}$ such that $t(a,n,\dots,n) = t(n,a,n,\dots,n) = \cdots = t(n, \dots, n,a)=n$.
\end{remark}

\subsection{Construction of an NU polymorphism} \label{sec:construction}

In this subsection we construct an NU polymorphism $f$ of the structure $\rstr{A}$ of arity $m^{2^n}+1$. The construction here coincides with the one in~\cite{Zh14}.

The operation $f$ is defined by
\[
f(\vc{x})=
\begin{cases}
n, & \text{if $m^{2^{n}}+1>m^{2^{n}}\cdot\less{n}{\vc{x}}$},\\
n-1, & \text{else if $\less{n}{\vc{x}}>m^{2^{n-1}}\cdot\less{n-1}{\vc{x}}$},\\
\vdots \\
r, & \text{else if $\less{r+1}{\vc{x}}>m^{2^{r}}\cdot\less{r}{\vc{x}}$},\\
\vdots \\
0, & \text{else if $\less{1}{\vc{x}}>m\cdot\less{0}{\vc{x}}$},\\
a, & \text{else},
\end{cases}
\]
where we use the same notation $\less{r}{\vc{x}}$ as in the previous subsection. 
A more compact way to define $f$ is by setting
\[
f(\vc{x})=
\begin{cases}
\max{F_{\vc{x}}}, & \text{if $F_{\vc{x}}\neq\nic$},\\
a, & \text{otherwise},
\end{cases}
\]
where
\[
F_{\vc{x}}=\{r\in\{0,\dots,n\}: \less{r+1}{\vc{x}}>m^{2^{r}}\cdot \less{r}{\vc{x}}\}.
\]
Note that the result of $f$ does not depend on the order of arguments, only on the number of occurrences of each element among the arguments.

We need to show that $f$ is an NU operation, which is compatible with the relations $S_0,\dots,S_n$ and each unary relation $X \subseteq A$.

First we show that $f$ is compatible with $X$. Let $\vc{x}\in A^{m^{2^{n}}+1}$ be a~tuple and $r$ be an element not in $\vc{x}$. If $r\in\{0,\dots,n\}$, then $\less{r+1}{\vc{x}}=\less{r}{\vc{x}}$, hence $\less{r+1}{\vc{x}}\leq m^{2^{r}}\cdot\less{r}{\vc{x}}$, and therefore $f(\vc{x})\neq r$. If $r=a$, then there exists the least $s\in \{1,\dots,n+1\}$ such that $\less{s}{\vc{x}}\neq 0$. So $\less{s}{\vc{x}}>m^{2^{s-1}}\cdot\less{s-1}{\vc{x}}=0$ and $f(\vc{x})\neq a$.

Next we verify that $f$ is an NU operation. We already know that $f(r,r,\dots, r)=r$ since it is compatible with ${\{r\}}$ for any $r\in A$. 
Let $s,t\in A$, ${s\neq t}$, and $\vc{x}=(s,\dots,s,t,s,\dots,s)$, where $t$ is at any position in $\vc{x}$. Compatibility of~$f$ with $\{s,t\}$ yields $f(\vc{x})\in{\{s,t\}}$. If $s>t$, then $s\neq a$ and \[
\less{s+1}{\vc{x}}=m^{2^n}+1>m^{2^s}=m^{2^s}\cdot\less{s}{\vc{x}},
\] hence $f(\vc{x})=s$. If $s<t$, then $t\neq a$ and \[
\less{t+1}{\vc{x}}=m^{2^{n}}+1<m^{2^{t}}\cdot m^{2^{n}}=m^{2^{t}}\cdot \less{t}{\vc{x}},
\]
hence $f(\vc{x})\neq t$ and, again, $f(\vc{x})=s$.

It remains to prove that $f$ is compatible with $S_k$ for every $k\in {\{0,\dots,n\}}$. To simplify the notation, let $l={m^{2^{n}}+1}$ denote the arity of $f$. Let $k\in {\{0,\dots,n\}}$ and $M$ be a matrix $(m+1)\times (m^{2^{n}}+1)$ of elements in $A$ such that $M_{\hv\,1}, M_{\hv\,2},\dots,M_{\hv\,l}\in S_k$, i.e. every column of $M$ is in $S_k$. Let us denote ${\vc{v}=(f(M_{1\,\hv}),\dots,f(M_{(m+1)\,\hv}))}$. We claim that $\vc{v} \in S_k$.
\[
	\begin{matrix}
 & & & & & \vc{v} \\
 & & & & & \raisebox{\depth}{\rotatebox{270}{$=$}} \\
m_{1\,1} & m_{1\,2} & \dots & m_{1\,l} & \overset{f}{\rightarrow} & f(M_{1\,\hv}) \\
m_{2\,1} & m_{2\,2} & \dots & m_{2\,l} & \overset{f}{\rightarrow} & f(M_{2\,\hv}) \\
\vdots & \vdots  &  & \vdots & \vdots & \vdots \\
m_{(m+1)\,1} & m_{(m+1)\,2} & \dots & m_{(m+1)\,l} & \overset{f}{\rightarrow} & f(M_{(m+1)\,\hv}) \\
\raisebox{\depth}{\rotatebox{270}{$\in$}} & \raisebox{\depth}{\rotatebox{270}{$\in$}} & & \raisebox{\depth}{\rotatebox{270}{$\in$}} & & \raisebox{\depth}{\rotatebox{270}{$\overset{\rotatebox{90}{?}}{\in}$}} \\
S_k & S_k & \dots & S_k & & S_k
	\end{matrix}
\]

Since the columns of M are in $S_k$, there is no $k$ in the first row of M. By compatibility with the unary relations we get also $f(M_{1\,\hv})\neq k$. In the other rows there are elements from ${\{a,k,k+1,\dots,n\}}$ and, again by compatibility with the unary relations, $f(M_{i\,\hv})$ is in the same set for every $i\in{\{2,\dots,m+1\}}$. If there is some $u>k$  in a column, then there is $u$ at all coordinates of that column. Hence there is the same number of $u$'s in every row for every $u>k$. This yields
\begin{equation} \label{kompS:01}
\less{u}{M_{1\,\hv}}=\less{u}{M_{2\,\hv}}=\dots=\less{u}{M_{(m+1)\,\hv}}
\end{equation}
for all $u\in{\{k+1,\dots,n+1\}}$.

It follows immediately that if the result of any row is $f(M_{i\,\hv})=u>k$, then $f(M_{1\,\hv})=\dots=f(M_{(m+1)\,\hv})=u$ and hence $\vc{v}\in S_k$.

Otherwise, we have $f(M_{i\,\hv})\leq k$ for every $i\in {\{1,\dots,m+1\}}$. In this case  $\vc{v}\in S_k$ except for the case $\vc{v}=(a,k,\dots,k)$. We will show that this, however, can never happen.

For a contradiction, suppose $f(M_{1\,\hv})=a$ and $f(M_{i\,\hv})=k$ for every $i\in{\{2,\dots,m+1\}}$. From $f(M_{1\,\hv})=a$, we get
\begin{multline*} 
\less{k}{M_{1\,\hv}} \leq m^{2^{k-1}}\cdot\less{k-1}{M_{1\,\hv}} \leq m^{2^{k-1}}\cdot m^{2^{k-2}}\cdot\less{k-2}{M_{1\,\hv}} \leq \dots \leq\\
\leq m^{2^{k-1}}\cdot m^{2^{k-2}}\cdot \dots \cdot m^{2^0}\cdot \less{0}{M_{1\,\hv}} = m^{2^{k-1}+2^{k-2}+\dots+2^0}\cdot \val{a}{M_{1\,\hv}} = \\
= m^{2^{k}-1}\cdot \val{a}{M_{1\,\hv}}.
\end{multline*}
Since there is no $k$ in the first row of $M$, it follows from (\ref{kompS:01}) that \[
\less{k+1}{M_{i\,\hv}}=\less{k+1}{M_{1\,\hv}}=\less{k}{M_{1\,\hv}}.
\]
From $f(M_{i\,\hv})=k$, we get \[
m^{2^k} \cdot \val{a}{M_{i\,\hv}} = m^{2^k} \cdot \less{k}{M_{i\,\hv}} < \less{k+1}{M_{i\,\hv}}
\]
for every $i\in{\{2,\dots,m+1\}}$.

Putting the previous three equations together we obtain \[
m^{2^{k}}\cdot \val{a}{M_{i\,\hv}}<m^{2^{k}-1}\cdot \val{a}{M_{1\,\hv}},
\]
which simplifies to
\begin{equation} \label{kompS:04}
m\cdot \val{a}{M_{i\,\hv}}<\val{a}{M_{1\,\hv}}
\end{equation}
for every $i\in {\{2,\dots,m+1\}}$. Therefore,
\begin{equation} \label{kompS:03}
\sum_{i=2}^{m+1}{\val{a}{M_{i\,\hv}}}<\val{a}{M_{1\,\hv}}.
\end{equation}

On the other hand, if we have the element $a$ in the first coordinate of a column, there must be $a$ also in some other coordinate of that column, otherwise the column would be $(a,k,\dots,k)\notin S_k$. Therefore, there is at least the same number of~occurrences of~$a$ in the first line as there is in the other lines combined, i.e. \[
\sum_{i=2}^{m+1}{\val{a}{M_{i\,\hv}}}\geq \val{a}{M_{1\,\hv}},
\]
which contradicts the inequality (\ref{kompS:03}). This proves our claim that $\vc{v}\in S_k$.

We have shown that $f$ is compatible with the relation $S_k$ for every $k\in\{0,\dots,n\}$, hence it is, indeed, an~NU~polymorphism of $\rstr{A}$ of arity $m^{2^n}+1$.

\section{Binary relations}

In this section we prove the lower bounds in Theorem~\ref{thm:lower} for $m=2$ and $n \geq 3$.

Similarly to the previous section, it will be convenient to change the parametrization: we fix $n \in \N_0$ and we will construct a relational structure $\rstr{B}$ with universe of size $n+3$ and relations of arity at most 2, which admits an NU polymorphism of arity $2^{2^n}+1$ and no NU polymorphism of smaller arity.

We remark that a folklore reduction (see, e.g., ~\cite[Proposition 3.1]{Bar13}) would also produce binary structures with quite large minimal arity of an NU polymorphism, but this construction is insufficient to match the lower bound $2^{2^n}$.

\subsection{Construction}

The argument illustrated in Subsection \ref{subsec:example} does not work for binary structures, because the number of $a$'s in the vectors $\vc{v}_k$ would not grow at all. This issue can be fixed by introducing additional elements that behave similarly to $a$. The optimal choice for our purposes is to work with one additional element.

We define
\[
\rstr{B}=(B;\R{1}{1},\R{1}{2},\R{2}{1},\R{2}{2},\dots,\R{n}{1},\R{n}{2},\family{X}{X\subseteq \B{n+1}}),
\]
where
\begin{align*}
B &= \{a_1,a_2,0,1, \dots,n\} \\
\R[n]{i}{j} & =\left(\{a_1,a_2,0,1,\dots,i-1\}\times\{a_1,a_2,i\}\ \setm\ \{(a_j,i)\}\right)\\
& \cup \{(i+1,i+1),\dots,(n,n)\}. 
\end{align*}

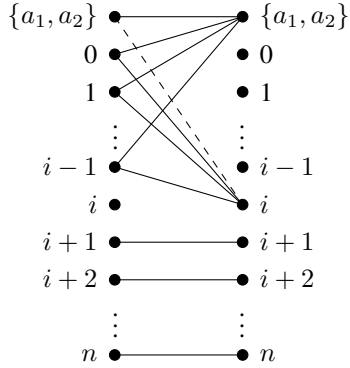
\MyFig{ 
	\centering
	\begin{subfigure}{.5\textwidth}
		\centering
		\begin{tikzpicture}[
bod/.style={radius=\relbod},
textL/.style={left=\reltext},
teckyL/.style={left=\reltecky},
textR/.style={right=\reltext},
teckyR/.style={right=\reltecky},
edge/.style={\reledge}
]

\foreach \x/\ozn in {0/a, 1/0, 2/1, 3/t1, 4/i-1, 5/i, 6/i+1, 7/i+2, 8/t2, 9/n}
	\coordinate (l\ozn) at (0,{\x*(\relkrok)});
\foreach \x/\ozn in {0/a, 1/0, 2/1, 3/t1, 4/i-1, 5/i, 6/i+1, 7/i+2, 8/t2, 9/n}
	\coordinate (r\ozn) at ({0+\relrozestup},{\x*(\relkrok)});

\filldraw[fill=black, draw=black]
(la) circle[bod] node[textL] {$\{a_1,a_2\}$}
(l0) circle[bod] node[textL] {0}
(l1) circle[bod] node[textL] {1}
(lt1) node[teckyL] {$\vdots$}
(li-1) circle[bod] node[textL] {$i-1$}
(li) circle[bod] node[textL] {$i$}
(li+1) circle[bod] node[textL] {$i+1$}
(li+2) circle[bod] node[textL] {$i+2$}
(lt2) node[teckyL] {$\vdots$}
(ln) circle[bod] node[textL] {$n$};

\filldraw[fill=black, draw=black]
(ra) circle[bod] node[textR] {$\{a_1,a_2\}$}
(r0) circle[bod] node[textR] {0}
(r1) circle[bod] node[textR] {1}
(rt1) node[teckyR] {$\vdots$}
(ri-1) circle[bod] node[textR] {$i-1$}
(ri) circle[bod] node[textR] {$i$}
(ri+1) circle[bod] node[textR] {$i+1$}
(ri+2) circle[bod] node[textR] {$i+2$}
(rt2) node[teckyR] {$\vdots$}
(rn) circle[bod] node[textR] {$n$};

\draw[edge,dashed] (la)--(ri);
\draw[edge] (l0)--(ri);
\draw[edge] (l1)--(ri);
\draw[edge] (li-1)--(ri);

\draw[edge] (la)--(ra);
\draw[edge] (l0)--(ra);
\draw[edge] (l1)--(ra);
\draw[edge] (li-1)--(ra);

\draw[edge] (li+1)--(ri+1);
\draw[edge] (li+2)--(ri+2);
\draw[edge] (ln)--(rn);

\draw ({(\relrozestup)/2},{10*(\relkrok)-.6}) node[text width=5cm] {
The dashed line denotes that $a_j$ is not in the relation, but ``the other $a$'' is.
};
		\end{tikzpicture}
	\caption{Relation $\R{i}{j}$ on $\B{n+1}$}
	\end{subfigure}
	\begin{subfigure}{.5\textwidth}
		\centering
		\begin{tikzpicture}[
bod/.style={radius=\relbod},
textL/.style={left=\reltext},
teckyL/.style={left=\reltecky},
textR/.style={right=\reltext},
teckyR/.style={right=\reltecky},
edge/.style={\reledge}
]

\foreach \x/\ozn in {-1/a1, 0/a2, 1/0, 2/1, 3/2, 4/3}
	\coordinate (l\ozn) at (0,{\x*(\relkrok)});
\foreach \x/\ozn in {-1/a1, 0/a2, 1/0, 2/1, 3/2, 4/3}
	\coordinate (r\ozn) at ({0+\relrozestup},{\x*(\relkrok)});

\filldraw[fill=black, draw=black]
(la1) circle[bod] node[textL] {$a_1$}
(la2) circle[bod] node[textL] {$a_2$}
(l0) circle[bod] node[textL] {0}
(l1) circle[bod] node[textL] {1}
(l2) circle[bod] node[textL] {2}
(l3) circle[bod] node[textL] {3};

\filldraw[fill=black, draw=black]
(ra1) circle[bod] node[textR] {$a_1$}
(ra2) circle[bod] node[textR] {$a_2$}
(r0) circle[bod] node[textR] {0}
(r1) circle[bod] node[textR] {1}
(r2) circle[bod] node[textR] {2}
(r3) circle[bod] node[textR] {3};

\draw[edge] (la2)--(r2);
\draw[edge] (l0)--(r2);
\draw[edge] (l1)--(r2);

\draw[edge] (la1)--(ra1);
\draw[edge] (la2)--(ra1);
\draw[edge] (l0)--(ra1);
\draw[edge] (l1)--(ra1);

\draw[edge] (la1)--(ra2);
\draw[edge] (la2)--(ra2);
\draw[edge] (l0)--(ra2);
\draw[edge] (l1)--(ra2);

\draw[edge] (l3)--(r3);

\draw ({(\relrozestup)/2},{5*(\relkrok)-.6}) node[text width=5cm] {
Note that the pair $(a_1,2)$ is not connected.
};
		\end{tikzpicture}
	\caption{A concrete example -- $\R{2}{1}$ on $B=\B{4}$}
	\end{subfigure}
\caption{Relation $\R{i}{j}$}
} 

Similarly as before we define $\B{i+1}=\{a_1,a_2,0,1,\ldots,i\}$ and equip $B_i$ with the linear order 
\[
a_1 < a_2 < 0 < 1 < \ldots < i.
\]
 
Each polymorphism of $\rstr{B}$ is compatible with the equivalences $\B{i}|i|\dots|n$ ($i\in\{1,\dots,n\}$). The proof is analogous to the proof of Lemma \ref{lem:congruence}. An analogue of Lemma \ref{lem:comp} is the following.

\begin{lemma} \label{lem:comp2}
Let $t$ be a polymorphism of $\rstr{B}$ and $i\in\{0,\dots,n\}$. Let $l_{a_1},l_{a_2},l_{i},l_{i+1},\dots,l_{n}\in \N_0$ be such that their sum is $k$ and $l_{a_1}+l_{a_2}\leq l_{i}$. Let $\Bmin{i}=\B{i}\setm \{a_1,a_2\}$. Then
\begin{align*}
t(a_1,\dots,a_1,a_2,\dots,a_2,\Bmin{i},\dots,\Bmin{i},i+1,\dots,i+1,\dots,n,\dots,n)&\neq a_1,\\
t(\underbrace{a_2,\dots,a_2}_{l_{a_1}+l_{a_2}},\underbrace{a_1,\dots,a_1}_{l_{a_1}+l_{a_2}},\underbrace{\Bmin{i},\dots,\Bmin{i}}_{l_{i}-(l_{a_1}+l_{a_2})},\underbrace{i+1,\dots,i+1}_{l_{i+1}},\dots,\underbrace{n,\dots,n}_{l_{n}})&\neq a_2
\end{align*}
whenever
\begin{align*}
t(\underbrace{a_1,\dots,a_1}_{l_{a_1}},\underbrace{a_2,\dots,a_2}_{l_{a_2}},\underbrace{i,\dots,i}_{l_{i}},\underbrace{i+1,\dots,i+1}_{l_{i+1}},\dots,\underbrace{n,\dots,n}_{l_{n}})&=i
\end{align*}
(for any permutation of the arguments, the same for every row).
\end{lemma}


\subsection{No low--arity NU polymorphism}

Similarly to Subsection~\ref{sec:non-existence} one shows that $\rstr{B}$ has no NU polymorphism of arity $2^{2^n}$. We discuss the necessary adjustment in the case $n=3$. 

Striving for a contradiction, let us assume that $t$ is an NU polymorphism of $\rstr{B}$ of arity $2^{2^3}=256$. Using the fact that $t$ is compatible with $\B{i}|i|\dots|n$  and Lemma \ref{lem:comp2} we show step by step that \[
t(\underbrace{a_1,\dots,a_1}_{128},\underbrace{a_2,\dots,a_2}_{128})\notin \{a_1,a_2\},
\]
which contradicts the compatibility of $t$ with the unary relation $\{a_1,a_2\}$. Similarly as in Subsection~\ref{sec:non-existence} we choose a suitable vector $\vc{w}_k$ in every step so that $t(\vc{w}_{k-1})\notin \{a_1,a_2\}$ implies $t(\vc{w}_k)\notin \{a_1,a_2\}$. In each step the number of $a_i$'s will be doubled.

Let us denote $\Bmin{3}=\B{3}\setm\{a_1,a_2\}={\{0,1,2\}}$. Since $t$ is NU, we have
\begin{align*}
t(a_1,3,3,\dots,3)&=3,\\
t(3,a_2,3,\dots,3)&=3,
\end{align*}
from which it follows by Lemma $\ref{lem:comp2}$ that
\begin{align*}
t(a_1,a_2,\Bmin{3},\dots,\Bmin{3})&\neq a_1,\\
t(a_1,a_2,\Bmin{3},\dots,\Bmin{3})&\neq a_2.
\end{align*}
In particular,
\begin{equation} \label{pr:bin0}
t(\vc{w}_0)=t(a_1,a_2,0,0,\underbrace{1,\dots,1}_{2^{4}-2^{2}=12},\underbrace{2,\dots,2}_{2^{8}-2^{4}})\notin \{a_1,a_2\}
\end{equation}
and similarly for any permutation of the arguments. We want to show
\begin{equation} \label{pr:bin1}
t(\vc{w}_1)=t(a_1,a_1,a_2,a_2,\underbrace{1,\dots,1}_{2^{4}-2^{2}=12},\underbrace{2,\dots,2}_{2^{8}-2^{4}})\notin \{a_1,a_2\}.
\end{equation}
If $t(\vc{w}_0)$ is equal to $1$ or $2$, then we get (\ref{pr:bin1}) using the compatibility with  $B_1|1|2|3$. Otherwise $t(\vc{w}_0)=0$ (we use the compatibility of $t$ with $\{a_1,a_2,0,1,2\}$) and using again the compatibility with $B_1|1|2|3$ we obtain
\begin{equation*}
t(\B{1},\B{1},\B{1},\B{1},1,\dots,1,2,\dots,2)\in \B{1}.
\end{equation*}
Since (\ref{pr:bin0}) holds for any permutation of arguments, then
\begin{align*}
t(a_1,a_2,0,0,1,\dots,1,2,\dots,2)&=0,\\
t(0,0,a_1,a_2,1,\dots,1,2,\dots,2)&=0,
\end{align*}
from which we deduce by Lemma~\ref{lem:comp2}
\begin{align*}
t(a_1,a_1,a_2,a_2,1,\dots,1,2,\dots,2)&\neq a_1,\\
t(a_1,a_1,a_2,a_2,1,\dots,1,2,\dots,2)&\neq a_2,
\end{align*}
so (\ref{pr:bin1}) holds. By the same argument, (\ref{pr:bin1}) also holds for any permutation of the arguments. Next we justify
\begin{equation} \label{pr:bin2}
t(\vc{w}_2)=t(a_1,a_1,a_1,a_1,a_2,a_2,a_2,a_2,\underbrace{0,\dots,0}_{8},\underbrace{2,\dots,2}_{2^8-2^4})\notin \{a_1,a_2\}.
\end{equation}

If $t(\vc{w}_1)=2$, then (\ref{pr:bin2}) is a consequence of the compatibility with $\B{2}|2|3$. Otherwise necessarily $t(\vc{w}_1)=1$ and using $\B{2}|2|3$ and (\ref{pr:bin1}) for suitable permutations we derive
\begin{align*}
t(a_1,a_1,a_2,a_2,\underbrace{1,\dots,1}_{12},2,\dots,2)&=1,\\
t(1,1,1,1,a_1,a_1,a_2,a_2,\underbrace{1,\dots,1}_{8},2,\dots,2)&=1,
\end{align*}
from which we deduce (\ref{pr:bin2}) using Lemma~\ref{lem:comp2}. In a similar way we can gradually prove the following:
\begin{align*}
t(\underbrace{a_1,\dots,a_1}_{2^3},\underbrace{a_2,\dots,a_2}_{2^3},\underbrace{2,\dots,2}_{2^8-2^4})\notin\{a_1,a_2\},\\
t(\underbrace{a_1,\dots,a_1}_{2^4},\underbrace{a_2,\dots,a_2}_{2^4},\underbrace{0,\dots,0}_{2^6-2^5},\underbrace{1,\dots,1}_{2^8-2^6})\notin\{a_1,a_2\},\\
t(\underbrace{a_1,\dots,a_1}_{2^5},\underbrace{a_2,\dots,a_2}_{2^5},\underbrace{1,\dots,1}_{2^8-2^6})\notin\{a_1,a_2\},\\
t(\underbrace{a_1,\dots,a_1}_{2^6},\underbrace{a_2,\dots,a_2}_{2^6},\underbrace{0,\dots,0}_{2^8-2^7})\notin\{a_1,a_2\},\\
t(\underbrace{a_1,\dots,a_1}_{2^7},\underbrace{a_2,\dots,a_2}_{2^7})\notin\{a_1,a_2\}.
\end{align*}

\subsection{NU polymorphism exists}

The following operation of arity $2^{2^n}+1$ is an NU polymorphism of $f$:
\[
f(\vc{x})=
\begin{cases}
n, & \text{if $2^{2^{n}}+1>2^{2^{n}}\cdot\lessB{n}{\vc{x}}$},\\
n-1, & \text{else if $\lessB{n}{\vc{x}}>2^{2^{n-1}}\cdot\lessB{n-1}{\vc{x}}$},\\
\vdots \\
r, & \text{else if $\lessB{r+1}{\vc{x}}>2^{2^{r}}\cdot\lessB{r}{\vc{x}}$},\\
\vdots \\
0, & \text{else if $\lessB{1}{\vc{x}}>2\cdot\lessB{0}{\vc{x}}$},\\
a_2, & \text{else if $\val{a_2}{\vc{x}}>\val{a_1}{\vc{x}}$},\\
a_1, & \text{else},
\end{cases}
\]
where $\lessB{r}{\vc{x}}$ has the same meaning as in Subsection~\ref{sec:non-existence}, i.e. $\lessB{r}{\vc{x}}$ is the number of occurrences of elements from $\B{r}$ in $\vc{x}$.

Operation $f$ is an NU operation compatible with all the unary relations. This can be verified the same way it was in Subsection~\ref{sec:construction}.

The compatibility of $f$ with $\R{k}{s}$ is also analogous. Let $k\in{\{0,\dots,n\}}$, $s\in{\{1,2\}}$. Let us denote $l=2^{2^n}$ and let $\vc{u}=(u_1,\dots,u_l)$, $\vc{w}=(w_1,\dots,w_l)$, where $u_i,w_i\in B$, $(u_i,w_i)\in\R{k}{s}$ for every $i\in{\{1,\dots,l\}}$. We claim that also $(f(\vc{u}),f(\vc{w}))\in\R{k}{s}$.
\[
	\begin{matrix}
u_1 & u_2 & \dots & u_l & \overset{f}{\rightarrow} & f(\vc{u}) \\
w_1 & w_2 & \dots & w_l & \overset{f}{\rightarrow} & f(\vc{w}) \\
\raisebox{\depth}{\rotatebox{270}{$\in$}} & \raisebox{\depth}{\rotatebox{270}{$\in$}} & & \raisebox{\depth}{\rotatebox{270}{$\in$}} & & \raisebox{\depth}{\rotatebox{270}{$\overset{\rotatebox{90}{?}}{\in}$}} \\
\R{k}{s} & \R{k}{s} & \dots & \R{k}{s} & & \R{k}{s}
	\end{matrix}
\]
The same way as in Subsection \ref{sec:construction} we observe that if $f(\vc{u})=r$ or $f(\vc{w})=r$ for some $r>k$, then $f(\vc{u})=f(\vc{w})=r$ and $(f(\vc{u}),f(\vc{w}))\in\R{k}{s}$.

Suppose that $f(\vc{u})\leq k$ and $f(\vc{w})\leq k$. Then the only possible outcome not in $\R{k}{s}$ is $f(\vc{u})=a_s$, $f(\vc{w})=k$. We show by contradiction this is impossible. The same reasoning we used in Subsection~\ref{sec:construction} to get (\ref{kompS:04}) is now used to get
\begin{equation} \label{kompR:04}
2\cdot\lessB{0}{\vc{w}}<\lessB{0}{\vc{u}}=\val{a_1}{\vc{u}}+\val{a_2}{\vc{u}}.
\end{equation}
Moreover, $\val{a_s}{\vc{u}}\leq\lessB{0}{\vc{w}}$ because if $u_i=a_s$ for some $i$, then either $w_i=a_1$ or~$w_i=a_2$ (since $(u_i,w_i)\in\R{k}{s}$). From this and (\ref{kompR:04}) we obtain\[
2\cdot\val{a_s}{\vc{u}}<\val{a_1}{\vc{u}}+\val{a_2}{\vc{u}}.
\]
We take $\val{a_s}{\vc{u}}$ away from both sides and see that the number of $a_s$ in~$\vc{u}$ is strictly lesser than the number of ``the other~$a$'s'', therefore $f(\vc{u})\neq a_s$, a contradiction with our assumption.


\bibliographystyle{plain}
\bibliography{CSPbib}

\end{document}